\documentclass[10pt]{article}
\usepackage{amsfonts}
\usepackage{amsmath}
\usepackage{amssymb}
\usepackage{amsthm}
\usepackage{enumerate}
\makeatletter
\def\imod#1{\allowbreak\mkern10mu({\operator@font mod}\,\,#1)}
\makeatother

\usepackage{setspace}

\newtheorem{theorem}{Theorem}[section]
\newtheorem{lemma}{Lemma}[section]
\newtheorem{corollary}{Corollary}[section]
\newtheorem{fact}{Fact}[section]
\newtheorem{conjecture}{Conjecture}[section]

\theoremstyle{definition}
\newtheorem{definition}{Definition}[section]
\newtheorem{remark}{Remark}[section]

\begin{document}
\begin{center}
\begin{singlespace}
\vskip 1cm{\LARGE\bf An Extension of the Abundancy Index to Certain Quadratic Rings
\vskip 1cm
\large
Colin Defant\footnote{This work was supported by National Science Foundation grant no. 1262930.}\footnote{Colin Defant\\ 18434 Hancock Bluff Rd. \\ Dade City, FL 33523}\\
Department of Mathematics\\
University of Florida\\
United States\\
cdefant@ufl.edu}
\end{singlespace}
\end{center}
\vskip .2 in

\begin{abstract}
We begin by introducing an extension of the traditional abundancy index to imaginary quadratic rings with unique factorization. After showing that many of the properties of the traditional abundancy index continue to hold in our extended form, we investigate what we call $n$-powerfully solitary numbers in these rings. This definition serves to extend the concept of solitary numbers, which have been defined and studied in the integers. We end with some open questions and a conjecture.    
\end{abstract}
\section{Introduction} 

Throughout this paper, we will let $\mathbb{N}$ denote the set of positive integers, and we will let $\mathbb{N}_0$ denote the set of nonnegative integers.
\par 
The arithmetic functions $\sigma_k$ are defined, for every integer $k$, by \\ 
$\displaystyle{\sigma_k(n)=\sum_{\substack{c\vert n\\c>0}}c^k}$, and it is conventional to write $\sigma_1=\sigma$. It is well-known that, for each integer $k\neq 0$, $\sigma_k$ is multiplicative and satisfies $\displaystyle{\sigma_k (p^\alpha)=\frac{p^{k(\alpha+1)}-1}{p^k-1}}$ for all (integer) primes $p$ and positive integers $\alpha$. The abundancy index of a positive integer $n$ is defined by $\displaystyle{I(n)=\frac{\sigma(n)}{n}}$. Using the formulas
$\displaystyle{\sigma(n)=\prod_{p^{\alpha}\parallel n}\frac{p^{\alpha+1}-1}{p-1}}$ and $\displaystyle{\sigma_{-1}(n)=\prod_{p^{\alpha}\parallel n}\frac{p^{-(\alpha+1)}-1}{p^{-1}-1}}$, it is easy to see that $I=\sigma_{-1}$. Some of the most common questions associated with the abundancy index are those related to friendly numbers. 
\par
Two or more distinct positive integers are said to be friends (with each other) if they have the same abundancy index.  For example, $I(6)=I(28)=I(496)=2$, so $6$, $28$, and $496$  are friends. A positive integer that has at least one friend is said to be friendly, and a positive integer that has no friends is said to be solitary. Clearly, $1$ is solitary as $I(n)>1=I(1)$ for any positive integer $n>1$. It is also not difficult to show, using the fact that $I=\sigma_{-1}$, that every prime power is solitary. 
In the next section, we extend the notions of the abundancy index and friendliness to imaginary quadratic integer rings that are also unique factorization domains. Observing the infinitude of possible such generalizations, we note four important properties of the traditional abundancy index that we wish to preserve (possibly with slight modifications). 
\begin{itemize}
\item The range of the function $I$ is a subset of the interval $[1,\infty)$.
\item If $n_1$ and $n_2$ are relatively prime positive integers, then \\ $I(n_1n_2)=I(n_1)I(n_2)$. 
\item If $n_1$ and $n_2$ are positive integers such that $n_1\vert n_2$, then $I(n_1)\leq I(n_2)$, with equality if and only if $n_1=n_2$. 
\item All prime powers are solitary. 
\end{itemize}
\par 
For any square-free integer $d$, let $\mathcal O_{\mathbb{Q}(\sqrt{d})}$ be the quadratic integer ring given by \[\mathcal O_{\mathbb{Q}(\sqrt{d})}=\begin{cases} \mathbb{Z}[\frac{1+\sqrt{d}}{2}], & \mbox{if } d\equiv 1\imod{4}; \\ \mathbb{Z}[\sqrt{d}], & \mbox{if } d\equiv 2, 3 \imod{4}. \end{cases}\] 
\par 
Throughout the remainder of this paper, we will work in the rings $\mathcal O_{\mathbb{Q}(\sqrt{d})}$ for different specific or arbitrary values of $d$. We will use the symbol ``$\vert$" to mean ``divides" in the ring $\mathcal O_{\mathbb{Q}(\sqrt{d})}$ in which we are working. 
Whenever we are working in a ring other than $\mathbb{Z}$, we will make sure to emphasize when we wish to state that one integer divides another in $\mathbb{Z}$. 
For example, if we are working in $\mathbb{Z}[i]$, the ring of Gaussian integers, we might say that $1+i\vert 1+3i$ and that $2\vert 6$ in $\mathbb{Z}$. We will also refer to primes in $\mathcal O_{\mathbb{Q}(\sqrt{d})}$ as ``primes," whereas we will refer to (positive) primes in $\mathbb{Z}$ as ``integer primes." 
Furthermore, we will henceforth focus exclusively on values of $d$ for which $\mathcal O_{\mathbb{Q}(\sqrt{d})}$ is a unique factorization domain and $d<0$. In other words, $d\in K$, where we will define $K$ to be the set $\{-163,-67,-43,-19,-11,-7,-3,-2,-1\}$. The set $K$ is known to be the complete set of negative values of $d$ for which $\mathcal O_{\mathbb{Q}(\sqrt{d})}$ is a unique factorization domain \cite{Stark67}.  
\par 
For now, let us work in a ring $\mathcal O_{\mathbb{Q}(\sqrt{d})}$ such that $d\!\in\! K$. For an element $a+b\sqrt{d}\in\mathcal O_{\mathbb{Q}(\sqrt{d})}$ with $a,b\in \mathbb{Q}$, we define the conjugate by $\overline{a+b\sqrt{d}}=a-b\sqrt{d}$. We also define the norm of an element $z$ by $N(z)=z\overline{z}$ and the absolute value of $z$ by $\vert z\vert=\sqrt{N(z)}$. From now on, we will assume familiarity with these objects and their properties (for example, $\overline{z_1z_2}=\overline{z_1}\hspace{0.75 mm}\overline{z_2}$ and $N(z)\in \mathbb{N}_0$), which are treated in Keith Conrad's online notes \cite{Conrad}. For $x,y\in\mathcal O_{\mathbb{Q}(\sqrt{d})}$, we say that $x$ and $y$ are associated, denoted $x\sim y$, if and only if $x=uy$ for some unit $u$ in the ring $\mathcal O_{\mathbb{Q}(\sqrt{d})}$. Furthermore, we will make repeated use of the following well-known facts. 
\begin{fact} \label{Fact1.1}
Let $d\!\in\! K$. If $p$ is an integer prime, then exactly one of the following is true. 
\begin{itemize}
\item $p$ is also a prime in $\mathcal O_{\mathbb{Q}(\sqrt{d})}$. In this case, we say that $p$ is inert in $\mathcal O_{\mathbb{Q}(\sqrt{d})}$. 
\item $p\sim \pi^2$ and $\pi\sim\overline{\pi}$ for some prime $\pi\in \mathcal O_{\mathbb{Q}(\sqrt{d})}$. In this case, we say $p$ ramifies (or $p$ is ramified) in $\mathcal O_{\mathbb{Q}(\sqrt{d})}$. 
\item $p=\pi\overline{\pi}$ and $\pi\not\sim\overline{\pi}$ for some prime $\pi\in\mathcal O_{\mathbb{Q}(\sqrt{d})}$. In this case, we say $p$ splits (or $p$ is split) in $\mathcal O_{\mathbb{Q}(\sqrt{d})}$.
\end{itemize}
\end{fact}
\begin{fact} \label{Fact1.2}
Let $d\!\in\! K$. If $\pi\!\in\!\mathcal O_{\mathbb{Q}(\sqrt{d})}$ is a prime, then exactly one of the following is true. 
\begin{itemize}
\item $\pi\sim q$ and $N(\pi)=q^2$ for some inert integer prime $q$. 
\item $\pi\sim\overline{\pi}$ and $N(\pi)=p$ for some ramified integer prime $p$. 
\item $\pi\not\sim\overline{\pi}$ and $N(\pi)=N(\overline{\pi})=p$ for some split integer prime $p$. 
\end{itemize}
\end{fact}
\begin{fact} \label{Fact1.3}
Let $\mathcal O_{\mathbb{Q}(\sqrt{d})}^*$ be the set of units in the ring $\mathcal O_{\mathbb{Q}(\sqrt{d})}$. Then $\mathcal O_{\mathbb{Q}(\sqrt{-1})}^*=\{\pm 1,\pm i\}$, $\displaystyle{\mathcal O_{\mathbb{Q}(\sqrt{-3})}^*=\left\{\pm 1,\pm \frac{1+\sqrt{-3}}{2},\pm \frac{1-\sqrt{-3}}{2}\right\}}$, and $\mathcal O_{\mathbb{Q}(\sqrt{d})}^*=\{\pm 1\}$ \\ 
whenever $d\in K\backslash\{-1,-3\}$. 
\end{fact}

\section{The Extension of the Abundancy Index}

For a nonzero complex number $z$, let $\arg (z)$ denote the argument, or angle, of $z$. We convene to write $\arg (z)\in [0,2\pi)$ for all $z\in\mathbb{C}$. For each $d\in K$, we define the set $A(d)$ by 
\[A(d)=\begin{cases} \{z\in\mathcal O_{\mathbb{Q}(\sqrt{d})} \backslash\{0\}: 0\leq \arg (z)<\frac{\pi}
{2}\}, & \mbox{if } d=-1; \\ \{z\in\mathcal O_{\mathbb{Q}(\sqrt{d})} \backslash\{0\}: 0\leq \arg (z)<\frac{\pi}
{3}\}, & \mbox{if } d=-3; \\ \{z\in\mathcal O_{\mathbb{Q}(\sqrt{d})} \backslash\{0\}: 0\leq \arg (z)<\pi\}, & \mbox{otherwise}. \end{cases}\] 
Thus, every nonzero element of $\mathcal O_{\mathbb{Q}(\sqrt{d})}$ can be written uniquely as a unit times a product of primes in $A(d)$. Also, every $z\in\mathcal O_{\mathbb{Q}(\sqrt{d})}\backslash\{0\}$ is associated to a unique element, which we will call $B(z)$, of $A(d)$. We are now ready to define analogues of the arithmetic functions $\sigma_k$. 
\begin{definition} \label{Def2.1}
Let $d\in K$, and let $n\in \mathbb{Z}$. 
Define the function 
\newline $\delta_n\colon\mathcal O_{\mathbb{Q}(\sqrt{d})}\backslash\{0\}\rightarrow [1,\infty)$ by 
\[\delta_n (z)=\sum_{\substack{x\vert z\\x\in A(d)}}\vert x \vert^n.\]
\end{definition}
\begin{remark} \label{Rem2.1}
We note that, for each $x$ in the summation in the above definition, we may cavalierly replace $x$ with one of its associates. This is because associated numbers have the same absolute value. In other words, the only reason for the criterion $x\!\in\! A (d)$ in the summation that appears in Definition \ref{Def2.1} is to forbid us from counting associated divisors as distinct terms in the summation, but we may choose to use any of the associated divisors as long as we only choose one. This should not be confused with how we count conjugate divisors (we treat $2+i$ and $2-i$ as distinct divisors of $5$ in $\mathbb{Z}[i]$ because $2+i\not\sim 2-i$).  
\end{remark}

\begin{remark} \label{Rem2.2}
We note that, by choosing different values of $d$, the functions $\delta_n$ change dramatically. For example, $\delta_2(3)=10$ when we work in the ring $\mathcal O_{\mathbb{Q}(\sqrt{-1})}$, but $\delta_2(3)=16$ when we work in the ring $\mathcal O_{\mathbb{Q}(\sqrt{-2})}$. Perhaps it would be more precise to write $\delta_n(z,d)$, but we will omit the latter component for convenience. We note that we will also use this convention with functions such as $I_n$ (which we will define soon). 
\end{remark}
\par 
We will say that a function $f\colon\mathcal O_{\mathbb{Q}(\sqrt{d})}\backslash\{0\}\!\rightarrow\!\mathbb{R}$ is multiplicative if $f(xy)=f(x)f(y)$ whenever $x$ and $y$ are relatively prime (have no nonunit common divisors). 
\begin{theorem} \label{Thm2.1}
Let $d\!\in\! K$, and let $f, g\colon\mathcal O_{\mathbb{Q}(\sqrt{d})}\backslash\{0\}\!\rightarrow\!\mathbb{R}$ be multiplicative functions such that $f(u)=g(u)=1$ for all units $u\in \mathcal O_{\mathbb{Q}(\sqrt{d})}^*$. 
Define \\ 
$F\colon\mathcal O_{\mathbb{Q}(\sqrt{d})}\backslash\{0\}\rightarrow\mathbb{R}$ by 
\[F(z)=\sum_{\substack{x,y\in A(d)\\xy\sim z}}f(x)g(y).\]
Then $F$ is multiplicative. 
\end{theorem}
\begin{proof}
Suppose $z_1, z_2\in\mathcal O_{\mathbb{Q}(\sqrt{d})}\backslash\{0\}$ and $\gcd (z_1, z_2)=1$. For any $x, y\in A(d)$ satisfying $xy\sim z_1z_2$, we may write $x=x_1x_2$, $y=y_1y_2$ so that $x_1y_1\sim z_1$ and $x_2y_2\sim z_2$. To make the choice of $x_1$, $x_2$, $y_1$, $y_2$ unique, we require $x_1, y_1\in A(d)$. Conversely, if we choose $x_1, x_2, y_1, y_2\in\mathcal O_{\mathbb{Q}(\sqrt{d})}\backslash\{0\}$ such that $x_1, y_1\!\in\! A(d)$, $x_1y_1\sim z_1$, $x_2y_2\sim z_2$, and $x_1x_2, y_1y_2\!\in\! A(d)$, then we may write $x=x_1x_2$ and $y=y_1y_2$ so that $xy\sim z_1z_2$. To simplify notation, write $B(x_2)=x_3$, $B(y_2)=y_3$, and let $C$ be the set of all ordered quadruples $(x_1, x_2, y_1, y_2)$ such that $x_1, x_2, y_1, y_2\in\mathcal O_{\mathbb{Q}(\sqrt{d})}\backslash\{0\}$, $x_1, y_1\in A(d)$, $x_1y_1\sim z_1$, $x_2y_2\sim z_2$, and $x_1x_2, y_1y_2\in A(d)$. We have established a bijection between $C$ and the set of ordered pairs $(x,y)$ satisfying $x, y\in A(d)$ and  $xy\sim z_1z_2$. Therefore, 
\[F(z_1z_2)=\sum_{\substack{x, y\in A(d)\\xy\sim z_1z_2}}f(x)g(y)=\sum_{(x_1, x_2, y_1, y_2)\in C}f(x_1x_2)g(y_1y_2)\]
\[=\sum_{(x_1, x_2, y_1, y_2)\in C}f(x_1)f(x_2)g(y_1)g(y_2)\]
\[=\sum_{(x_1, x_2, y_1, y_2)\in C}f(x_1)f(B(x_2))g(y_1)g(B(y_2))\]
\[=\sum_{\substack{x_1, y_1\in A(d)\\x_1y_1\sim z_1}}f(x_1)g(y_1)\sum_{\substack{x_3, y_3\in A(d)\\x_3y_3\sim z_2}}f(x_3)g(y_3)=F(z_1)F(z_2).\]
\end{proof}
\begin{corollary} \label{Cor2.1}
For any integer $n$, $\delta_n$ is multiplicative.
\end{corollary} 
\begin{proof}
Noting that $\delta_n(w_1)=\delta_n(w_2)$ whenever $w_1\sim w_2$, we may let
\newline $f, g\colon\mathcal O_{\mathbb{Q}(\sqrt{d})}\backslash\{0\}\rightarrow\mathbb{R}$ be the functions defined by $f(z)=\vert z\vert ^n$ and $g(z)=1$ for all $z\in\mathcal O_{\mathbb{Q}(\sqrt{d})}\backslash\{0\}$. Then the desired result follows immediately from Theorem \ref{Thm2.1}.  
\end{proof} 
\begin{definition} \label{Def2.2}
For each positive integer $n$, define the function \\
$I_n\colon\mathcal O_{\mathbb{Q}(\sqrt{d})}\backslash\{0\}\rightarrow[1,\infty)$ by $\displaystyle{I_n(z)=\frac{\delta_n(z)}{\vert z\vert ^n}}$.  
We say that two or more numbers $z_1, z_2, \ldots, z_r\in\mathcal O_{\mathbb{Q}(\sqrt{d})}\backslash\{0\}$ are \textit{$n$-powerfully friendly} (or \textit{$n$-powerful friends}) \textit{in $\mathcal O_{\mathbb{Q}(\sqrt{d})}$} if $I_n(z_j)=I_n(z_k)$ and $\vert z_j\vert\neq\vert z_k\vert$ for all distinct $j, k\in \{1, 2, \ldots, r\}$. 
Any $z\in\mathcal O_{\mathbb{Q}(\sqrt{d})}\backslash\{0\}$ that has no $n$-powerful friends in $\mathcal O_{\mathbb{Q}(\sqrt{d})}$ is said to be \textit{$n$-powerfully solitary in $\mathcal O_{\mathbb{Q}(\sqrt{d})}$}. 
\end{definition}
\begin{remark} \label{Rem2.3}
Whenever $n=1$, we will omit the adjective ``$1$-powerfully" in the preceding definitions.   
\end{remark}
As an example, we will let $d=-1$ so that $\mathcal O_{\mathbb{Q}(\sqrt{d})}=\mathbb{Z}[i]$. Let us compute $I_2(9+3i)$. We have $9+3i=3(1+i)(2-i)$, so $\delta_2(9+3i)=N(1)+N(3)+N(1+i)+N(2-i)+N(3(1+i))+N(3(2-i))+N((1+i)(2-i))+N(3(1+i)(2-i))=1+9+2+5+18+45+10+90=180$. Then $\displaystyle{I_2(9+3i)=\frac{180}{N(3(1+i)(2-i))}=2}$. Although $I_2(3+9i)$ is also equal to $2$, $3+9i$ and $9+3i$ are not $2$-powerful friends in $\mathbb{Z}[i]$ because $\vert 3+9i\vert=\vert 9+3i\vert$.
We now establish some important properties of the functions $I_n$.  
\begin{theorem} \label{Thm2.2}
Let $n\!\in\!\mathbb{N}$, $d\!\in\! K$, and $z_1, z_2, \pi\in\mathcal O_{\mathbb{Q}(\sqrt{d})}\backslash\{0\}$ with $\pi$ a prime. Then, if we are working in the ring $\mathcal O_{\mathbb{Q}(\sqrt{d})}$, the following statements are true. 
\begin{enumerate}[(a)] 
\item The range of $I_n$ is a subset of the interval $[1,\infty)$, and $I_n(z_1)=1$ if and only if $z_1$ is a unit in $\mathcal O_{\mathbb{Q}(\sqrt{d})}$. If $n$ is even, then $I_n(z_1)\in\mathbb{Q}$. 
\item $I_n$ is multiplicative.  
\item $I_n(z_1)=\delta_{-n}(z_1)$. 
\item If $z_1\vert z_2$, then $I_n(z_1)\leq I_n(z_2)$, with equality if and only if $z_1\sim z_2$. 
\item If $z_1\sim \pi ^k$ for a nonnegative integer $k$, then $z_1$ is $n$-powerfully solitary in $\mathcal O_{\mathbb{Q}(\sqrt{d})}$. 
\end{enumerate}  	
\end{theorem} 
\begin{proof}
The first sentence in part $(a)$ is fairly clear, and the second sentence becomes equally clear if one uses the fact that $\vert z_1\vert^n\!\!\in\mathbb{N}$ whenever $n$ is even. To prove part $(b)$, suppose that $z_1$ and $z_2$ are relatively prime elements of $\mathcal O_{\mathbb{Q}(\sqrt{d})}$. Then, by Corollary \ref{Cor2.1}, $\displaystyle{I_n(z_1z_2)=\frac{\delta_n(z_1z_2)}{\vert z_1z_2\vert ^n}=\frac{\delta_n(z_1)\delta_n(z_2)}{\vert z_1\vert ^n\vert z_2\vert ^n}=I_n(z_1)I_n(z_2)}$. In order to prove part $(c)$, it suffices, due to the truth of part $(b)$, to prove that $I_n(\pi ^\alpha)=\delta_{-n}(\pi ^\alpha)$ for any prime $\pi$ and nonnegative integer $\alpha$. To do so is fairly routine, as
\[I_n(\pi ^\alpha)=\frac{\delta_n(\pi ^\alpha)}{\vert\pi ^\alpha\vert^n} =\frac{\sum_{j=0}^\alpha\vert \pi ^j\vert ^n}{\vert \pi ^\alpha\vert ^n}=\sum_{j=0}^\alpha\vert \pi ^{j-\alpha}\vert ^n\] \[=\sum_{j=0}^\alpha\vert\pi ^{\alpha-j}\vert ^{-n}=\sum_{l=0}^\alpha\vert \pi ^l\vert ^{-n}=\delta_{-n}(\pi ^\alpha).\]
The truth of statement $(d)$ follows from part $(c)$ because, if $z_1\vert z_2$, then  
\[I_n(z_2)=\delta_{-n}(z_2)=\sum_{\substack{x\vert z_2\\x\in A(d)}}\vert x\vert ^{-n}\]
\[=\sum_{\substack{x\vert z_1\\x\in A(d)}}\vert x\vert ^{-n}+\sum_{\substack{x\vert z_2\\x\nmid z_1\\x\in A(d)}}\vert x\vert ^{-n}=I_n(z_1)+\sum_{\substack{x\vert z_2\\x\nmid z_1\\x\in A(d)}}\vert x\vert ^{-n}.\]
\par 
Finally, for part $(e)$, we provide a proof for the case when $n$ is even. We postpone the proof for the case in which $n$ is odd until the next section. Let $\pi$ be a prime in $\mathcal O_{\mathbb{Q}(\sqrt{d})}$, and suppose that $z_1\sim \pi ^k$  for a nonnegative integer $k$. If $k=0$, then $z_1$ is a unit and the result follows from part $(a)$. Therefore, assume $k>0$. Assume, for the sake of finding a contradiction, that $I_n(z_1)=I_n(z_2)$ and $\vert z_1\vert\neq\vert z_2\vert$ for some $z_2\in\mathcal O_{\mathbb{Q}(\sqrt{d})}\backslash\{0\}$. Under this assumption, we have $\vert z_2\vert ^n\delta_n(z_1)=\vert z_1\vert ^n\delta_n(z_2)$. Either $N(\pi)=p$ is an integer prime or $N(\pi)=q^2$, where $q$ is an integer prime.
\par 
First, suppose $N(\pi)=p$ is an integer prime. Then the statement 
\newline $\vert z_2\vert ^n\delta_n(z_1)=\vert z_1\vert ^n\delta_n(z_2)$ is equivalent to $N(z_2)^{n/2}\delta_n(\pi ^k)=p^{kn/2}\delta_n(z_2)$. Noting that $N(z_2)^{n/2}$, $\delta_n(\pi ^k)$, and $\delta_n(z_2)$ are integers (because $n$ is even) and that $p\nmid\delta_n(\pi ^k)=1+p^{n/2}+\cdots+p^{kn/2}$ in $\mathbb{Z}$, we find $p^{kn/2}\vert N(z_2)^{n/2}$ in $\mathbb{Z}$. This implies that $p^k\vert N(z_2)$ in $\mathbb{Z}$, and we conclude that there exist nonnegative integers $t_1, t_2$ satisfying $\pi ^{t_1}\overline{\pi}^{t_2}\vert z_2$ and $t_1+t_2=k$. If $\pi\sim\overline{\pi}$, then we have $\pi ^k\vert z_2$, from which part $(d)$ yields the desired contradiction. Otherwise, $\pi$ and $\overline{\pi}$ are relatively prime, so we may use parts $(b)$ and $(d)$ to write
\[I_n(z_2)\geq I_n(\pi ^{t_1})I_n(\overline{\pi}^{t_2})=\frac{1+p^{n/2}+\cdots+p^{t_1n/2}}{p^{t_1n/2}}\frac{1+p^{n/2}+\cdots+p^{t_2n/2}}
{p^{t_2n/2}}\]
\[=\frac{(1+p^{n/2}+\cdots+p^{t_1n/2})(1+p^{n/2}+\cdots+p^{t_2n/2} )}{p^{kn/2}}\]
\[\geq \frac{1+p^{n/2}+\cdots+p^{kn/2}}{p^{kn/2}}=I_n(\pi ^k)=I_n(z_2).\]
This implies that $I_n(z_2)=I_n(\pi ^{t_1}\overline{\pi}^{t_2})$, from which part $(d)$ tells us that 
$z_2\sim \pi ^{t_1}\overline{\pi}^{t_2}$. Therefore, $\vert z_2\vert=\vert\pi ^{t_1}\overline{\pi}^{t_2}\vert=\sqrt{p}^{t_1+t_2}=\sqrt{p}^k=\vert\pi^k\vert=\vert z_1\vert$, which we assumed was false. 
\par 
Now, suppose that $N(\pi)=q^2$, where $q$ is an integer prime ($q$ is inert). Then the statement $\vert z_2\vert ^n\delta_n(z_1)=\vert z_1 \vert ^n\delta_n(z_2)$ is equivalent to 
\newline $N(z_2)^{n/2}\delta_n(\pi ^k)=q^{kn}
\delta_n(z_2)$. As before, $N(z_2)^{n/2}$, $\delta_n(\pi ^k)$, and $\delta_n(z_2)$ are integers, and $q\nmid \delta_n(\pi ^k)=1+q^n+\cdots+q^{kn}$ in $\mathbb{Z}$. Therefore, $q^{kn}\vert N(z_2)^{n/2}$ in $\mathbb{Z}$, so $q^{2k}\vert N(z_2)$ in $\mathbb{Z}$. As $q$ is inert, this implies that $q^k\vert z_2$, so $z_1\vert z_2$ (note that $z_1\sim \pi ^k\sim q^k$). Therefore, part $(d)$ provides the final contradiction, and the proof is complete. 
\end{proof}
It is much easier to deal with the functions $I_n$ when $n$ is even than when $n$ is odd because, when $n$ is even, the values of $\delta_n(z)$ and $\vert z\vert ^n$ are positive integers. Therefore, we will devote the next section to developing an understanding of the functions $I_n$ for odd values of $n$. 
\section{When $n$ is Odd}
We begin by establishing some definitions and lemmata that will later prove themselves useful. Let $W$ be the set of all square-free positive integers, and write $W=\{w_0, w_1, w_2, \ldots\}$ so that $w_0=1$ and $w_i<w_j$ for all nonnegative 
integers $i<j$. Let $F$ be the set of all finite linear combinations of elements of $W$ with rational coefficients. That is, $F=\{a_0+a_1\sqrt{w_1}+\cdots+a_m\sqrt{w_m}: a_0, a_1, \ldots, a_m\!\in\!\mathbb{Q}, m\!\in\!\mathbb{N}_0\}$. For any $r\!\in\! F$, the choice of the rational coefficients is unique. More formally, if 
$a_0+a_1\sqrt{w_1}+\cdots+a_m\sqrt{w_m}=b_0+b_1\sqrt{w_1}+\cdots+b_m\sqrt{w_m}$, where $a_0, a_1, \ldots, a_m, b_0, b_1, \ldots, b_m\in\mathbb{Q}$, then $a_i=b_i$ for all $i\in \{0, 1, \ldots, m\}$ \cite{Klazar09}. Note that $F$ is a subfield of the real numbers. 
\begin{definition} \label{Def3.1}
For $r\!\in\! F$ and $j\in\mathbb{N}_0$, let $C_j(r)$ be the unique rational coefficient of $\sqrt{w_j}$ in the expansion of $r$. That is, the sequence $(C_j(r))_{j=0}^\infty$ is the unique infinite sequence of rational numbers that has finitely many nonzero terms and that satisfies $\displaystyle{r=\sum_{j=0}^\infty C_j(r)\sqrt{w_j}}$. 
\end{definition}
As an example, $\displaystyle{C_5\left(\frac{3}{5}-\sqrt{6}+\frac{1}{3}\sqrt{7}\right)=\frac{1}{3}}$ because $w_5=7$. 
\begin{definition} \label{Def3.2}
Let $p$ be an integer prime. For $r\in F$, we say that \textit{$r$ has a $\sqrt{p}$ part} if there exists some positive integer $j$ such that $C_j(r)\neq 0$ and $p\vert w_j$ (in $\mathbb{Z}$). We say that $r$ does not have a $\sqrt{p}$ part if no such positive integer $j$ exists. 
\end{definition}
For example, if $\displaystyle{r=\frac{1}{2}+3\sqrt{10}}$, then $r$ has a $\sqrt{2}$ part, and $r$ has a $\sqrt{5}$ part. However, $\displaystyle{\frac{1}{2}+3\sqrt{10}}$ does not have a $\sqrt{7}$ part.  
\begin{lemma} \label{Lem3.1}
If $r_1, r_2\in F$ each do not have a $\sqrt{p}$ part for some integer prime $p$, then $r_1r_2$ does not have a $\sqrt{p}$ part.
\end{lemma} 
\begin{proof}
Suppose $p\vert w_j$ for some positive integer $j$. Then, if we let $SF(n)$ denote the square-free part of an integer $n$ and consider the basic algebra used to multiply elements of $F$, we find that
\[C_j(r_1r_2)=\sum_{\substack{i_1, i_2\in\mathbb{N}_0\\SF(w_{i_1}w_{i_2})=w_j}}C_{i_1}(r_1)C_{i_2}(r_2)\sqrt{\frac{w_{i_1}w_{i_2}}{w_j}}.\]
For every pair of nonnegative integers $i_1, i_2$ satisfying $SF(w_{i_1}w_{i_2})=w_j$, either $p\vert w_{i_1}$ or $p\vert w_{i_2}$. This implies that either $C_{i_1}(r_1)=0$ or $C_{i_2}(r_2)=0$ by the hypothesis that each of $r_1$ and $r_2$ does not have a $\sqrt{p}$ part. Thus, $C_j(r_1r_2)=0$. As $w_j$ was an arbitrary square-free positive integer divisible by $p$, we conclude that $r_1r_2$ does not have a $\sqrt{p}$ part. 
\end{proof}
\begin{lemma} \label{Lem3.2}
If each of $r_1, r_2, \ldots, r_l\in F$ does not have a $\sqrt{p}$ part for some integer prime $p$, then $r_1r_2\cdots r_l$ does not have a $\sqrt{p}$ part. 
\end{lemma}
\begin{proof}
The desired result follows immediately from repeated use of Lemma \ref{Lem3.1}. 
\end{proof}
\begin{lemma} \label{Lem3.3}
If $r_1\in F$ has a $\sqrt{p}$ part and $r_2\in F\backslash\{0\}$ does not have a $\sqrt{p}$ part for some integer prime $p$, then $r_1r_2$ has a $\sqrt{p}$ part. 
\end{lemma}
\begin{proof}
Write $\displaystyle{r_1=r_3+\sum_{i=1}^ka_i\sqrt{x_i}}$, where $r_3\in F$ does not have a $\sqrt{p}$ part and, for all distinct $i, j\in \{1, 2, \ldots, k\}$, we have $a_i\in\mathbb{Q}\backslash\{0\}$, $x_i\in W$, $p\vert x_i$ in $\mathbb{Z}$, and $x_i\neq x_j$. If we write $\displaystyle{v_i=\frac{x_i}{p}}$ for all $i\in \{1, 2, \ldots, k\}$, then each $v_i$ is a square-free positive integer that is not divisible by $p$. Therefore, \\ 
$\displaystyle{r_1r_2=\left(r_3+\sqrt{p}\sum_{i=1}^ka_i\sqrt{v_i}\right)r_2=r_2r_4\sqrt{p}+r_2r_3}$, where $\displaystyle{r_4=\sum_{i=1}^ka_i\sqrt{v_i}}$. By the hypothesis that $r_1$ has a $\sqrt{p}$ part, $r_4\neq 0$. As each of $r_2, r_4$ is nonzero and does not have a $\sqrt{p}$ part, Lemma \ref{Lem3.1} guarantees that $r_2r_4$ is nonzero and does not have a $\sqrt{p}$ part. Now, it is easy to see that this implies that $\sqrt{p}r_2r_4$ has a $\sqrt{p}$ part. Furthermore, each of $r_2, r_3$ does not have a $\sqrt{p}$ part, so Lemma \ref{Lem3.1} tells us that $r_2r_3$ does not have a $\sqrt{p}$ part. Thus, it is clear that $r_2r_4\sqrt{p}+r_2r_3$ has a $\sqrt{p}$ part, so the proof is complete. 
\end{proof}
\begin{lemma} \label{Lem3.4}
Let us fix $d\in K$ and work in the ring $\mathcal O_{\mathbb{Q}(\sqrt{d})}$. Let $\pi$ be a prime such that $N(\pi)=p$ is an integer prime. If $n$ is an odd positive integer and $\pi\vert z$ for some $z\in\mathcal O_{\mathbb{Q}(\sqrt{d})}\backslash\{0\}$, then $I_n(z)\in F$ and $I_n(z)$ has a $\sqrt{p}$ part. 
\end{lemma}
\begin{proof}
It is clear that $I_n(z)\in F$ (this is also true for positive even integer values of $n$). 
Write $\displaystyle{z\sim \pi ^\alpha\overline{\pi}^\beta\prod_{j=1}^r\pi_j^{\alpha_j}}$, where, for all distinct $j, k\in\{1, 2, \ldots, r\}$, $\pi_j$ is prime, $N(\pi_j)\neq p$, $\alpha_j$ is a positive integer, 
and $\pi_j\not\sim\pi_k$. Fix some $j\in\{1, 2, \ldots, r\}$. If $\pi_j$ is associated to an inert integer prime, then $I_n(\pi_j^{\alpha_j})\in\mathbb{Q}$, so $I_n(\pi_j^{\alpha_j})$ does not have a $\sqrt{p}$ part. If $N(\pi_j)=p_0$ for some 
integer prime $p_0$, then $I_n(\pi_j^{\alpha_j})=a+b\sqrt{p_0}$ for some $a, b\in\mathbb{Q}$. Again, we conclude that $I_n(\pi_j^{\alpha_j})$ does not have a $\sqrt{p}$ part because $p_0\neq p$. Writing $\displaystyle{x=\prod_{j=1}^r\pi_j^{\alpha_j}}$, Lemma \ref{Lem3.2} and the multiplicativity of $I_n$ guarantee that $I_n(x)$ does not have a $\sqrt{p}$ part. We now consider two cases. 
\par 
First, consider the case in which $p$ ramifies in $\mathcal O_{\mathbb{Q}(\sqrt{d})}$ (meaning $\pi\sim \overline{\pi}$). Then $z\sim \pi^{\alpha+\beta}x$. Using part $(c)$ of Theorem \ref{Thm2.2}, we have $I_n(\pi^{\alpha+\beta})=\delta_{-n}(\pi^{\alpha+\beta})=$ \\ 
$\displaystyle{\sum_{m=0}^{\alpha+\beta}\frac{1}{\vert\pi^m\vert ^n}=\sum_{m=0}^{\alpha+\beta}\frac{1}{\sqrt{p}^{mn}}=t_1+t_2\sqrt{p}}$, where $t_1$ and $t_2$ are positive rational numbers. Thus, $I_n(\pi^{\alpha+\beta})$ has a $\sqrt{p}$ part, so Lemma \ref{Lem3.3} guarantees that $I_n(z)$ has a $\sqrt{p}$ part. 
\par 
Next, consider the case in which $p$ splits in $\mathcal O_{\mathbb{Q}(\sqrt{d})}$ (meaning $\pi\not\sim\overline{\pi}$). Then we have $\displaystyle{I_n(\pi^\alpha\overline{\pi}^\beta)=\delta_{-n}(\pi^\alpha)\delta_{-n}(\overline{\pi}^\beta)=\left(\sum_{m=0}^\alpha\frac{1}{\vert \pi^m\vert ^n}\right)\left(\sum_{m=0}^\beta\frac{1}{\vert \overline{\pi}^m\vert ^n}\right)=}$ \\
$\displaystyle{\left(\sum_{m=0}^\alpha\frac{1}{\sqrt{p}^{mn}}\right)\left(\sum_{m=0}^\beta\frac{1}{\sqrt{p}^{mn}}\right)=(u_1+u_2\sqrt{p})(u_3+u_4\sqrt{p})}$, where $u_1, u_2, u_3, u_4$ \\ 
are positive rational numbers. Then $(u_1+u_2\sqrt{p})(u_3+u_4\sqrt{p})=u_1u_3+pu_2u_4+(u_1u_4+u_2u_3)\sqrt{p}$. As $u_1u_4+u_2u_3>0$, $I_n(\pi^\alpha\overline{\pi}^\beta)$ has a $\sqrt{p}$ part. Once again, Lemma \ref{Lem3.3} guarantees that $I_n(z)$ has a $\sqrt{p}$ part. 
\end{proof}
\begin{lemma} \label{Lem3.5}
Let $p$ be an integer prime, and let $m_1, m_2, \beta_1, \beta_2$ be nonnegative integers satisfying $(p^{m_1}+p^{m_2})(p^{\beta_1+\beta_2+1}+1)=(p^{\beta_1}+p^{\beta_2})(p^{m_1+m_2+1}+1)$. Then either $m_1=\beta_1$ and $m_2=\beta_2$ or $m_1=\beta_2$ and $m_2=\beta_1$.
\end{lemma}
\begin{proof}
Without loss of generality, we may write $m_1=\min(m_1, m_2, \beta_1, \beta_2)$. We may also assume that $\beta_1\leq\beta_2$ so that it suffices to show that $m_1=\beta_1$ and $m_2=\beta_2$. Dividing each side of the given equation by $p^{m_1}$, we have 
\begin{equation} \label{Eq3.1}
(1+p^{m_2-m_1})(p^{\beta_1+\beta_2+1}+1)=(p^{\beta_1-m_1}+p^{\beta_2-m_1})(p^{m_1+m_2+1}+1).
\end{equation}
Suppose $m_1=\beta_1$. Then \eqref{Eq3.1} becomes $(1+p^{m_2-m_1})(p^{m_1+\beta_2+1}+1)=(1+p^{\beta_2-m_1})(p^{m_1+m_2+1}+1)$. Now, define a function $f\colon\mathbb{R}\rightarrow\mathbb{R}$ by  
\\ 
$\displaystyle{f(x)=\frac{p^{m_1+x+1}+1}{1+p^{x-m_1}}}$. We may differentiate to get 
\[f'(x)=\frac{(p^{m_1+x})(p^{2m_1+1}-1)}{(p^x+p^{m_1})^2}\log{p}>0,\]
so $f$ is one-to-one. As $f(m_2)=f(\beta_2)$, we have $m_2=\beta_2$. Therefore, we only need to show that $m_1=\beta_1$. 
\par 
Suppose $p\neq 2$. Then, if $m_1<\beta_1$, we may read \eqref{Eq3.1} modulo $p$ to reach a contradiction. Thus, if $p\neq 2$, we are done. Now, suppose $p=2$ and $m_1<\beta_1$ so that \eqref{Eq3.1} becomes 
\begin{equation} \label{Eq3.2}
(1+2^{m_2-m_1})(2^{\beta_1+\beta_2+1}+1)=(2^{\beta_1-m_1}+2^{\beta_2-m_1})(2^{m_1+m_2+1}+1).
\end{equation}
The right-hand side of \eqref{Eq3.2} is even, which implies that we must have $m_1=m_2$ so that $1+2^{m_2-m_1}=2$. Dividing each side of \eqref{Eq3.2} by $2$ yields $2^{\beta_1+\beta_2+1}+1=(2^{\beta_1-m_1-1}+2^{\beta_2-m_1-1})(2^{2m_1+1}+1)$. As the left-hand side of this last equation is odd, we must have $\beta_1=m_1+1$. Therefore, 
$2^{\beta_1+
\beta_2+1}+1=(1+2^{\beta_2-\beta_1})(2^{2\beta_1-1}+1)=2^{\beta_1+\beta_2-1}+2^{\beta_2-\beta_1}+2^{2\beta_1-1}+1$. If we subtract $2^{\beta_1+\beta_2-1}+1$ from each side of this last equation, we get $3\cdot2^{\beta_1+\beta_2-1}=2^{\beta_2-\beta_1}+2^{2\beta_1-1}$. However, $3\cdot2^{\beta_1+\beta_2-1}>2^{\beta_1+\beta_2-1}+2^{\beta_1+\beta_2-1}>2^{\beta_2-\beta_1}+2^{2\beta_1-1}$, so we have reached our final contradiction. This completes the proof. 
\end{proof}
We now possess the tools necessary to complete the proof of part $(e)$ of Theorem \ref{Thm2.2}. We do so in the following two theorems. 
\begin{theorem} \label{Thm3.1}
Let us work in a ring $\mathcal O_{\mathbb{Q}(\sqrt{d})}$ with $d\in K$, and let $n$ be an odd positive integer. Let $\pi$ be a prime such that $\pi\sim\overline{\pi}$, and let $k$ be a positive integer. Then $\pi^k$ is $n$-powerfully solitary in $\mathcal O_{\mathbb{Q}(\sqrt{d})}$. 
\end{theorem}
\begin{proof}
We suppose, for the sake of finding a contradiction, that there exists $x\!\in\!\mathcal O_{\mathbb{Q}(\sqrt{d})}\backslash\{0\}$ such that $\vert x\vert\neq\vert\pi^k\vert$ and $I_n(x)=I_n(\pi^k)$. Suppose that $\pi_0$ is a prime such that $\pi_0\vert x$ and $N(\pi_0)=p_0$ is an integer prime. Then, by Lemma \ref{Lem3.4}, $I_n(x)$ has a $\sqrt{p_0}$ part. This implies that $I_n(\pi^k)$ has a $\sqrt{p_0}$ part. However, if $N(\pi)\!=\!p$, where $p$ is an integer prime, then $I_n(\pi^k)=\displaystyle{\sum_{m=0}^k\frac{1}{\sqrt{p}^{mn}}=t_1+t_2\sqrt{p}}$ for some $t_1, t_2\in\mathbb{Q}$. Hence, we find that $p_0=p$, which means that $\pi_0\sim\pi$. On the other hand, if $\pi$ is associated to an inert integer prime $q$, then $I_n(\pi^k)\in\mathbb{Q}$. Therefore, if a prime that is not associated to $\pi$  divides $x$, that prime must be associated to an inert integer prime. We now consider two cases.
\par 
Case 1: In this case, $\pi\sim q$, where $q$ is an inert integer prime. This implies that all primes dividing $x$ must be associated to inert integer primes, so $\delta_n(x)$ and $\vert x\vert$ are integers. From $I_n(x)=I_n(\pi^k)$ and $\vert\pi^k\vert^n=q^{kn}$, we have $\delta_n(x)q^{kn}=\delta_n(\pi^k)\vert x\vert^n$. We know that $\displaystyle{\delta_n(\pi^k)=\sum_{j=0}^k\vert\pi^j\vert^n=1+\sum_{j=1}^kq^{jn}}$, so $q\nmid\delta_n(\pi^k)$ in $\mathbb{Z}$. Therefore, $q^{kn}$ divides $\vert x\vert^n$ in $\mathbb{Z}$, so $q^k$ divides $\vert x\vert$ in $\mathbb{Z}$. We conclude that $q^k\vert x$, so $\pi^k\vert x$. However, part $(d)$ of Theorem \ref{Thm2.2} tells us that this is a contradiction. 
\par 
Case 2: In this case, $N(\pi)=p$ is an integer prime. Because all of the prime divisors of $x$ that are not associated to $\pi$ must be associated to inert integer 
primes, we may write $\displaystyle{x\sim\pi^\alpha\prod_{j=1}^tq_j^{\beta_j}}$, where $\alpha\in\mathbb{N}_0$ and, for each $j\in\{1, 2, \ldots, t\}$, $q_j$ is an inert integer prime and $\beta_j$ is a positive integer. 
Note that $\alpha\geq 1$ because $I_n(\pi^k)$ has a $\sqrt{p}$ part, which implies that $I_n(x)$ has a $\sqrt{p}$ part. Also, $\alpha<k$ because, otherwise, $\pi^k\vert x$, from which part $(d)$ of Theorem \ref{Thm2.2} yields a contradiction. We have 
\[I_n(\pi^k)=\frac{\sum_{l=0}^k\sqrt{p}^{ln}}{\sqrt{p}^{kn}}=\frac{\sqrt{p}^{(k+1)n}-1}{\sqrt{p}^{kn}(\sqrt{p}^n-1)},\]
and 
\[I_n(\pi^\alpha)=\frac{\sum_{l=0}^\alpha\sqrt{p}^{ln}}{\sqrt{p}^{\alpha n}} =\frac{\sqrt{p}^{(\alpha+1)n}-1}{\sqrt{p}^{\alpha n}(\sqrt{p}^n-1)}.\]
Now, $\displaystyle{\frac{I_n(\pi^k)}{I_n(\pi^\alpha)}=I_n\left(\prod_{j=1}^tq_j^{\beta_j}\right)\in\mathbb{Q}}$ because each integer prime $q_j$ is inert. This implies that $\displaystyle{(p^{(\alpha+1)n}-1)\frac{I_n(\pi^k)}{I_n(\pi^\alpha)}\in\mathbb{Q}}$. We have 
\[(p^{(\alpha+1)n}-1)\frac{I_n(\pi^k)}{I_n(\pi^\alpha)}=(p^{(\alpha+1)n}-1)\frac{\sqrt{p}^{(k+1)n-1}}{(\sqrt{p}^{(\alpha+1)n}-1)\sqrt{p}^{(k-\alpha)n}}\]
\[=(\sqrt{p}^{(\alpha+1)n}-1)(\sqrt{p}^{(\alpha+1)n}+1)\frac{\sqrt{p}^{(k+1)n}-1}{(\sqrt{p}^{(\alpha+1)n}-1)\sqrt{p}^{(k-\alpha)n}}\] 
\[=\frac{(\sqrt{p}^{(k+1)n}-1)(\sqrt{p}^{(\alpha+1)n}+1)}{\sqrt{p}^{(k-\alpha)n}}\in\mathbb{Q}.\] If $k$ is odd, then $\sqrt{p}^{(k+1)n}-1$ is 
rational, which implies that $\alpha$ must also be odd. Similarly, if $\alpha$ is odd, then $k$ must be odd. Therefore, $k$ and $\alpha$ have the same parities, which 
implies that $\sqrt{p}^{(k-\alpha)n}$ is rational. This implies $(\sqrt{p}^{(k
+1)n}-1)(\sqrt{p}^{(\alpha+1)n}+1)\in
\mathbb{Q}$. We clearly have a contradiction if $k$ and $\alpha$ are both even, so they must both be odd. As $k$ is odd, we have 
\[I_n(\pi^k)=\delta_{-n}(\pi^k)=\sum_{l=0}^k\frac{1}{\sqrt{p}^{ln}}=\left(\sum_{m=0}^{\frac{k-1}{2}}\frac{1}{\sqrt{p}^{2mn}}\right)+\left(\sum_{m=0}^{\frac{k-1}{2}}\frac{1}{\sqrt{p}^{2mn}}\right)\left(\frac{1}{\sqrt{p}^n}\right)\] 
\[=\left(\sum_{m=0}^{\frac{k-1}{2}}\frac{1}{p^{mn}}\right)\left(1+\frac{1}{\sqrt{p}^n}\right)=\frac{h_1}{p^n-1}\left(1+\frac{1}{\sqrt{p}^n}\right),\]
where 
\[h_1=\left(\sum_{m=0}^{\frac{k-1}{2}}\frac{1}{p^{mn}}\right)\left(p^n-1\right)=\frac{p^{\frac{k+1}{2}n}-1}{p^{\frac{k-1}{2}n}}.\]
Similarly, if we write $\displaystyle{h_2=\frac{p^{\frac{\alpha+1}{2}n}-1}{p^{\frac{\alpha-1}{2}n}}}$, then we have 
\\ 
$I_n(\pi^\alpha)=\displaystyle{\frac{h_2}{p^n-1}\left(1+\frac{1}{\sqrt{p}^n}\right)}$. Now,
\[I_n\left(\prod_{j=1}^tq_j^{\beta_j}\right)=\frac{I_n(\pi^k)}{I_n(\pi^\alpha)}=\frac{h_1}{h_2}=\frac{p^{\frac{k+1}{2}n}-1}{p^{\frac{k-\alpha}{2}n}\left(p^{\frac{\alpha+1}{2}n}-1\right)},\]
so
\[\left[\delta_n\left(\prod_{j=1}^tq_j^{\beta_j}\right)\right]\left[p^{\frac{k-\alpha}{2}n}\right]\left[p^{\frac{\alpha+1}{2}n}-1\right]=\left[\left\lvert\prod_{j=1}^tq_j^{\beta_j}\right\rvert^n\right]\left[p^{\frac{k+1}{2}n}-1\right].\] Notice that each bracketed expression in this last equation is an integer, and notice that $p$ divides the left-hand side in $\mathbb{Z}$. However, $p$ does not divide the right-hand side in $\mathbb{Z}$, so we have a contradiction. 
\end{proof}
We now only have to prove part $(e)$ of Theorem \ref{Thm2.2}
for the case in which $n$ is odd and $\pi\not\sim\overline{\pi}$. We do so as a corollary of the following more general theorem. 
\begin{theorem} \label{Thm3.2}
Let us work in a ring $\mathcal O_{\mathbb{Q}(\sqrt{d})}$ with $d\in K$, and let $n$ be an odd positive integer. Let $\pi$ be a prime such that $\pi\not\sim\overline{\pi}$, and let $k_1, k_2$ be nonnegative integers. Then $\pi^{k_1}\overline{\pi}^{k_2}$ is $n$-powerfully solitary in $\mathcal O_{\mathbb{Q}(\sqrt{d})}$ unless, possibly, if $k_1$ and $k_2$ are both odd. In the case that $k_1$ and $k_2$ are both odd, any friend of $\pi^{k_1}\overline{\pi}^{k_2}$, say $x$, must satisfy $\displaystyle{x\sim\pi^{\alpha_1}\overline{\pi}^{\alpha_2}\prod_{j=1}^tq_j^{\gamma_j}}$ , where $\alpha_1, \alpha_2$ are odd positive integers and, for each $j\in\{1, 2, \ldots, t\}$, $q_j$ is an inert integer prime and $\gamma_j$ is a positive integer. 
\end{theorem} 
\begin{proof}
First note that Fact \ref{Fact1.2} tells us that $N(\pi)=N(\overline{\pi})=p$, where $p$ is an integer prime. 
\par 
We suppose, for the sake of finding a contradiction, that there exists $x\in\mathcal O_{\mathbb{Q}(\sqrt{d})}\backslash\{0\}$ such that $\vert x\vert\neq\vert\pi^{k_1}\overline{\pi}^{k_2}\vert$ and $I_n(x)=I_n(\pi^{k_1}\overline{\pi}^{k_2})$. Suppose that $\pi_0$ is a prime such that $\pi_0\vert x$ and $N(\pi_0)=p_0$ is an integer prime. Then, by Lemma \ref{Lem3.4}, $I_n(x)$ has a $\sqrt{p_0}$ part. This implies that $I_n(\pi^{k_1}\overline{\pi}^{k_2})$ has a $\sqrt{p_0}$ part. However, as $N(\pi)=N(\overline{\pi})=p$, we must have $I_n(\pi^{k_1}\overline{\pi}^{k_2})=I_n(\pi^{k_1})I_n(\overline{\pi}^{k_2})=\displaystyle{\left(\sum_{m=0}^{k_1}\frac{1}{\sqrt{p}^{mn}}\right)\left(\sum_{m=0}^{k_2}\frac{1}{\sqrt{p}^{mn}}\right)}=t_1+t_2\sqrt{p}$ for some $t_1, t_2\in\mathbb{Q}$, so we find that $p_0=p$. Therefore, if a prime that is not associated to $\pi$ or $\overline{\pi}$ divides $x$, that prime must be associated to an inert integer prime. Hence, we may write $\displaystyle{x\sim\pi^{\alpha_1}\overline{\pi}^{\alpha_2}
\prod_{j=1}^tq_j^{\gamma_j}}$, where $\alpha_1, \alpha_2\in\mathbb{N}_0$ and, for each $j\in\{1, 2, \ldots, t\}$, $q_j$ is an inert integer prime and $\gamma_j$ is a positive integer. 
\par
We have
\[I_n(\pi^{k_1})=\frac{\sum_{l=0}^{k_1}\sqrt{p}^{ln}}{\sqrt{p}^{k_1n}}=\frac{\sqrt{p}^{(k_1+1)n}-1}{\sqrt{p}^{k_1n}(\sqrt{p}^n-1)},\] 
\[I_n(\overline{\pi}^{k_2})=\frac{\sum_{l=0}^{k_2}\sqrt{p}^{ln}}{\sqrt{p}^{k_2n}}=\frac{\sqrt{p}^{(k_2+1)n}-1}{\sqrt{p}^{k_2n}(\sqrt{p}^n-1)},\] 
\[I_n(\pi^{\alpha_1})=\frac{\sum_{l=0}^{\alpha_1}\sqrt{p}^{ln}}{\sqrt{p}^{\alpha_1n}}=\frac{\sqrt{p}^{(\alpha_1+1)n}-1}{\sqrt{p}^{\alpha_1n}(\sqrt{p}^n-1)},\] 
and 
\[I_n(\overline{\pi}^{\alpha_2})=\frac{\sum_{l=0}^{\alpha_2}\sqrt{p}^{ln}}{\sqrt{p}^{\alpha_2n}}=\frac{\sqrt{p}^{(\alpha_2+1)n}-1}{\sqrt{p}^{\alpha_2n}(\sqrt{p}^n-1)}.\] 
\par 
Now, $\displaystyle{\frac{I_n(\pi^{k_1})I_n(\overline{\pi}^{k_2})}{I_n(\pi^{\alpha_1})I_n(\overline{\pi}^{\alpha_2})}=\frac{I_n(\pi^{k_1}\overline{\pi}^{k_2})}{I_n(\pi^{\alpha_1}\overline{\pi}^{\alpha_2})}=I_n\left(\prod_{j=1}^tq_j^{\gamma_j}\right)\in\mathbb{Q}}$ because each integer prime $q_j$ is inert. This implies that \\ 
$\displaystyle{(p^{(\alpha_1+1)n}-1)(p^{(\alpha_2+1)n}-1)\frac{I_n(\pi^{k_1})I_n(\overline{\pi}^{k_2})}{I_n(\pi^{\alpha_1})I_n(\overline{\pi}^{\alpha_2})}\in\mathbb{Q}}$. We have 
\[(p^{(\alpha_1+1)n}-1)(p^{(\alpha_2+1)n}-1)\frac{I_n(\pi^{k_1})I_n(\overline{\pi}^{k_2})}{I_n(\pi^{\alpha_1})I_n(\overline{\pi}^{\alpha_2})}\]
\[=(p^{(\alpha_1+1)n}-1)(p^{(\alpha_2+1)n}-1)\frac{(\sqrt{p}^{(k_1+1)n}-1)(\sqrt{p}^{(k_2+1)n}-1)}{(\sqrt{p}^{(\alpha_1+1)n}-1)(\sqrt{p}^{(\alpha_2+1)n}-1)\sqrt{p}^{(k_1+k_2-\alpha_1-\alpha_2)n}}\]
\[=\frac{(\sqrt{p}^{(k_1+1)n}-1)(\sqrt{p}^{(k_2+1)n}-1)(\sqrt{p}^{(\alpha_1+1)n}+1)(\sqrt{p}^{(\alpha_2+1)n}+1)}{\sqrt{p}^{(k_1+k_2-\alpha_1-\alpha_2)n}}\in\mathbb{Q}.\]
\par 
We now consider several cases. In what follows, we will write \\ 
$\displaystyle{m_1=\frac{(k_1+1)n-1}{2}}$, $\displaystyle{m_2=\frac{(k_2+1)n-1}{2}}$, $\displaystyle{\beta_1=\frac{(\alpha_1+1)n-1}{2}}$, and \\
$\displaystyle{\beta_2=\frac{(\alpha_2+1)n-1}{2}}$. This will simplify notation because, for example, if $k_1$ is even, then $\sqrt{p}^{(k_1+1)n}=p^{m_1}\sqrt{p}$  and $m_1$ is a nonnegative integer.
\vspace{5 mm}
\\
Case 1: $\alpha_1\not\equiv\alpha_2\equiv k_1\equiv k_2\equiv 1\imod{2}$. In this case, $(\sqrt{p}^{(k_1+1)n}-1)(\sqrt{p}^{(k_2+1)n}-1)(\sqrt{p}^{(\alpha_2+1)n}+1)\in\mathbb{Q}$, so $\displaystyle{\frac{\sqrt{p}^{(\alpha_1+1)n}+1}{\sqrt{p}^{(k_1+k_2-\alpha_1-\alpha_2)n}}\in\mathbb{Q}}$. However, this is impossible because ($\alpha_1+1)n$ is odd. By the same argument, we may show that it is impossible to have exactly one of $k_1, k_2, \alpha_1, \alpha_2$ be even. 
\vspace{5 mm}
\\
Case 2: $\alpha_1\not\equiv\alpha_2\equiv k_1\equiv k_2\equiv 0\imod{2}$. In this case, $\sqrt{p}^{(\alpha_1+1)n}-1\in\mathbb{Q}$, and $\sqrt{p}^{(k_1+k_2-\alpha_1-\alpha_2)n}=\mu\sqrt{p}$ for some $\mu\in\mathbb{Q}$. This implies that \[(\sqrt{p}^{(k_1+1)n}-1)(\sqrt{p}^{(k_2+1)n}-1)(\sqrt{p}^{(\alpha_2+1)n}+1)\] 
\[=(p^{m_1}\sqrt{p}-1)(p^{m_2}\sqrt{p}-1)(p^{\beta_2}\sqrt{p}+1)=\lambda\sqrt{p}\] for some $\lambda\in\mathbb{Q}$. We may expand to get 
\[(p^{m_1}\sqrt{p}-1)(p^{m_2}\sqrt{p}-1)(p^{\beta_2}\sqrt{p}+1)\] 
\[=((p^{m_1+m_2+1}+1)-(p^{m_1}+p^{m_2})\sqrt{p})(p^{\beta_2}\sqrt{p}+1)\] 
\[=(p^{m_1+m_2+1}+1-p^{\beta_2+1}(p^{m_1}+p^{m_2}))+(p^{\beta_2}(p^{m_1+m_2+1}+1)-(p^{m_1}+p^{m_2}))\sqrt{p}.\] As $m_1, m_2, \beta_2\in\mathbb{N}_0$, we find that $p^{m_1+m_2+1}+1-p^{\beta_2+1}(p^{m_1}+p^{m_2})$ and $p^{\beta_2}(p^{m_1+m_2+1}+1)-(p^{m_1}+p^{m_2})$ are integers. Therefore, from the equation $(p^{m_1+m_2+1}+1-p^{\beta_2+1}(p^{m_1}+p^{m_2}))+(p^{\beta_2}(p^{m_1+m_2+1}+1)-(p^{m_1}+p^{m_2}))\sqrt{p}=\lambda\sqrt{p}$, we have $p^{m_1+m_2+1}+1-p^{\beta_2+1}(p^{m_1}+p^{m_2})=0$. Reading this last equation modulo $p$, we have a contradiction. The same argument eliminates the case $\alpha_2\not\equiv\alpha_1\equiv k_1\equiv k_2\equiv 0\imod{2}$.
\vspace{5 mm}
\\
Case 3: $k_1\not\equiv k_2\equiv\alpha_1\equiv\alpha_2\equiv 0\imod{2}$. In this case, $\sqrt{p}^{(k_1+1)n}-1\in\mathbb{Q}$, and $\sqrt{p}^{(k_1+k_2-\alpha_1-\alpha_2)n}=\mu\sqrt{p}$ for some $\mu\in\mathbb{Q}$. 
This implies that 
\[(\sqrt{p}^{(k_2+1)n}-1)(\sqrt{p}^{(\alpha_1+1)n}+1)(\sqrt{p}^{(\alpha_2+1)n}+1)\] 
\[=(p^{m_2}\sqrt{p}-1)(p^{\beta_1}\sqrt{p}+1)(p^{\beta_2}\sqrt{p}+1)=\lambda\sqrt{p}\] for some $\lambda\in\mathbb{Q}$. 
We may expand just as we did in Case 2, and we will find $p^{m_1+m_2+1}+1+p^{\beta_2+1}(p^{m_1}+p^{m_2})=0$, which is clearly a contradiction. This same argument eliminates the case $k_2\not\equiv k_1\equiv\alpha_1\equiv\alpha_2\equiv 0\imod{2}$.
\vspace{5 mm}
\\ 
Case 4: $k_1\equiv k_2\equiv 1\imod{2}$, and $\alpha_1\equiv\alpha_2\equiv 0\imod{2}$. In this case, $\displaystyle{\frac{(\sqrt{p}^{(k_1+1)n}-1)(\sqrt{p}^{(k_2+1)n}-1)}{\sqrt{p}^{(k_1+k_2-\alpha_1-\alpha_2)n}}\in\mathbb{Q}}$, so we must have \\
$(\sqrt{p}^{(\alpha_1+1)n}+1)(\sqrt{p}^{(\alpha_2+1)n}+1)=(p^{\beta_1}\sqrt{p}+1)(p^{\beta_2}\sqrt{p}+1)\in\mathbb{Q}$. However, this is impossible because $\beta_1$ and $\beta_2$ are nonnegative integers. 
\vspace{5 mm}
\\ 
Case 5: $k_1\equiv k_2\equiv 0\imod{2}$, and $\alpha_1\equiv\alpha_2\equiv 1\imod{2}$. In this case, $\displaystyle{\frac{(\sqrt{p}^{(\alpha_1+1)n}+1)(\sqrt{p}^{(\alpha_2+1)n}+1)}{\sqrt{p}^{(k_1+k_2-\alpha_1-\alpha_2)n}}\in\mathbb{Q}}$, so we must have 
\\ 
$(\sqrt{p}^{(k_1+1)n}-1)(\sqrt{p}^{(k_2+1)n}-1)=(p^{m_1}\sqrt{p}-1)(p^{m_2}\sqrt{p}-1)\in\mathbb{Q}$. However, this is impossible because $m_1$ and $m_2$ are nonnegative integers. 
\vspace{5 mm}
\\ 
Case 6: $k_1\equiv k_2\equiv\alpha_1\equiv\alpha_2\equiv 0\imod{2}$. In this case, $\sqrt{p}^{(k_1+k_2-\alpha_1-\alpha_2)n}\in\mathbb{Q}$, so 
\[(\sqrt{p}^{(k_1+1)n}-1)(\sqrt{p}^{(k_2+1)n}-1)(\sqrt{p}^{(\alpha_1+1)n}+1)(\sqrt{p}^{(\alpha_2+1)n}+1)\] 
\[=(p^{m_1}\sqrt{p}-1)(p^{m_2}\sqrt{p}-1)(p^{\beta_1}\sqrt{p}+1)(p^{\beta_2}\sqrt{p}+1)\in\mathbb{Q}.\] 
One may verify that, after expanding this last expression and noting that $m_1$, $m_2$, $\beta_1$, and $\beta_2$ must be positive integers, we arrive at the requirement $(p^{m_1}+p^{m_2})(p^{\beta_1+\beta_2+1}+1)=(p^{\beta_1}+p^{\beta_2})(p^{m_1+m_2+1}+1)$. Lemma \ref{Lem3.5} then guarantees that either $m_1=\beta_1$ and $m_2=\beta_2$ or $m_1=\beta_2$ and $m_2=\beta_1$, which means that either $k_1=\alpha_1$ and $k_2=\alpha_2$ or $k_1=\alpha_2$ and $k_2=\alpha_1$. Then $\displaystyle{\frac{I_n(\pi^{k_1})I_n(\overline{\pi}^{k_2})}{I_n(\pi^{\alpha_1})I_n(\overline{\pi}^{\alpha_2})}=I_n\left(\prod_{j=1}^tq_j^{\gamma_j}\right)=1}$, which implies that $\displaystyle{\prod_{j=1}^tq_j^{\gamma_j}}$ is a unit. However, we then find that $\vert\pi^{k_1}\overline{\pi}^{k_2}\vert=\vert\pi^{\alpha_1}\overline{\pi}^{\alpha_2}\vert=\vert x\vert$, which we originally assumed was not true. Therefore, this case yields a contradiction. 
\vspace{5 mm}
\\ 
Case 7: $k_1\equiv\alpha_1\equiv 1\imod{2}$ and $k_2\equiv\alpha_2\equiv 0\imod{2}$. In this case, $\displaystyle{\frac{(\sqrt{p}^{(k_1+1)n}-1)(\sqrt{p}^{(\alpha_1+1)n}+1)}{\sqrt{p}^{(k_1+k_2-\alpha_1-\alpha_2)n}}\in\mathbb{Q}}$, so we must have 
\\ 
$(\sqrt{p}^{(k_2+1)n}-1)(\sqrt{p}^{(\alpha_2+1)n}+1)=(p^{m_2}\sqrt{p}-1)(p^{\beta_2}\sqrt{p}+1)\in\mathbb{Q}$. Writing $(p^{m_2}\sqrt{p}-1)(p^{\beta_2}\sqrt{p}+1)=(p^{m_2+\beta_2+1}-1)+(p^{m_2}-p^{\beta_2})\sqrt{p}$ and noting that $m_2$ and $\beta_2$ are 
nonnegative integers, we find that $m_2=\beta_2$. Therefore, $k_2=\alpha_2$, so $\displaystyle{I_n\left(\prod_{j=1}^tq_j^{\gamma_j}\right)=\frac{I_n(\pi^{k_1})I_n(\overline{\pi}^{k_2})}{I_n(\pi^{\alpha_1})I_n(\overline{\pi}^{\alpha_2})}=\frac{I_n(\pi^{k_1})}{I_n(\pi^{\alpha_1})}}$. Because $\displaystyle{I_n\left(\prod_{j=1}^tq_j^{\gamma_j}\right)>1}$, we see that $\alpha_1<k_1$. As $k_1$ is odd, we have 
\[I_n(\pi^{k_1})=\delta_{-n}(\pi^{k_1})=\sum_{l=0}^{k_1}\frac{1}{\sqrt{p}^{ln}} =\left(\sum_{r=0}^{\frac{k_1-1}{2}}\frac{1}{\sqrt{p}^{2rn}}\right)+\left(\sum_{r=0}^{\frac{k_1-1}{2}}\frac{1}{\sqrt{p}^{2rn}}\right)\left(\frac{1}{\sqrt{p}^n}\right)\]
\[=\left(\sum_{r=0}^{\frac{k_1-1}{2}}\frac{1}{p^{rn}}\right)\left(1+\frac{1}{\sqrt{p}^n}\right)=\frac{h_1}{p^n-1}\left(1+\frac{1}{\sqrt{p}^n}\right),\]
where 
\[h_1=\left(\sum_{r=0}^{\frac{k_1-1}{2}}\frac{1}{p^{rn}}\right)\left(p^n-1\right)=\frac{p^{\frac{k_1+1}{2}n}-1}{p^{\frac{k_1-1}{2}n}}.\]
Similarly, if we write $\displaystyle{h_2=\frac{p^{\frac{\alpha_1+1}{2}n}-1}{p^{\frac{\alpha_1-1}{2}n}}}$, then we have 
\\ 
$\displaystyle{I_n(\pi^{\alpha_1})=\frac{h_2}{p^n-1}(1+\frac{1}{\sqrt{p}^n})}$. Now,
\[I_n\left(\prod_{j=1}^tq_j^{\gamma_j}\right)=\frac{I_n(\pi^{k_1})}{I_n(\pi^{\alpha_1})}=\frac{h_1}{h_2}=\frac{p^{\frac{k_1+1}{2}n}-1}{p^{\frac{k_1-\alpha_1}{2}n}(p^{\frac{\alpha_1+1}{2}n}-1)},\]
so 
\[\left[\delta_n\left(\prod_{j=1}^tq_j^{\gamma_j}\right)\right]\left[p^{\frac{k_1-\alpha_1}{2}n}\right]\left[p^{\frac{\alpha_1+1}{2}n}-1\right]=\left[\left\lvert\prod_{j=1}^tq_j^{\gamma_j}\right\rvert^n\right]\left[p^{\frac{k_1+1}{2}n}-1\right].\] 
Now, each bracketed part of this last equation is an integer, and $p$ divides the left-hand side in $\mathbb{Z}$. However, $p$ does not divide the right-hand side in $\mathbb{Z}$, so we have a contradiction. We may use this same argument to find contradictions in the three other cases in which $k_1\not\equiv k_2\imod{2}$ and $\alpha_1\not\equiv \alpha_2\imod{2}$. 
\par 
One may check that we have found contradictions for all of the possible choices of parities of $k_1$, $k_2$, $\alpha_1$, and $\alpha_2$ except the case in which all four are odd. Therefore, the proof is complete. 
\end{proof}
\begin{corollary} \label{Cor3.1}
Let $d\in K$, and let $k, n\in\mathbb{N}$ with $n$ odd. If $\pi$ is a prime in $\mathcal O_{\mathbb{Q}(\sqrt{d})}$ such that $\pi\not\sim\overline{\pi}$, then $\pi^k$ is $n$-powerfully solitary in $\mathcal O_{\mathbb{Q}(\sqrt{d})}$. 
\end{corollary}
\begin{proof}
Setting $k_1=k$ and $k_2=0$ in Theorem \ref{Thm3.1}, we find that $\pi^k$ is $n$-powerfully solitary in $\mathcal O_{\mathbb{Q}(\sqrt{d})}$ because $k_2$ is even. 
\end{proof}
\begin{corollary} \label{Cor3.2}
Let $d\!\in\! K$, and let $p$ be an integer prime. Let $k$ be a positive integer that is either even or equal to $1$, and let $n$ be an odd positive integer. If $z\sim p^k$, then $z$ is $n$-powerfully solitary in $\mathcal O_{\mathbb{Q}(\sqrt{d})}$. 
\end{corollary}
\begin{proof}
If $p$ is inert or ramified in $\mathcal O_{\mathbb{Q}(\sqrt{d})}$, then $z\sim\pi^\alpha$ for some prime $\pi$ and some positive integer $\alpha$. Therefore, $z$ is $n$-powerfully solitary in $\mathcal O_{\mathbb{Q}(\sqrt{d})}$ by part $(e)$ of Theorem \ref{Thm2.2}. If $p$ splits in $\mathcal O_{\mathbb{Q}(\sqrt{d})}$ and $k=1$, then $z\sim\pi\overline{\pi}$. Therefore, by Theorem \ref{Thm3.2}, any friend of $z$, say $x$, must satisfy $\displaystyle{x\sim\pi^{\alpha_1}\overline{\pi}^{\alpha_2}\prod_{j=1}^tq_j^{\gamma_j}}$ , where $\alpha_1, \alpha_2$ are odd positive integers and, for each $j\in\{1, 2, \ldots, t\}$, $q_j$ is an inert integer prime and $\gamma_j$ is a positive integer. However, this implies that $\pi\overline{\pi}\vert x$, so $z\vert x$. This is a contradiction. Finally, if $p$ splits in $\mathcal O_{\mathbb{Q}(\sqrt{d})}$ and $k$ is even, then $z\sim\pi^k\overline{\pi}^k$. As $k$ is even, the result follows from Theorem \ref{Thm3.2}. 
\end{proof}
\par 	
Note that Corollary \ref{Cor3.1} delivers the final blow in the proof of part $(e)$ of Theorem \ref{Thm2.2}.
\section{Concluding Remarks and Open Questions} 
After the introduction of our generalization of the abundancy index, we quickly become inundated with new questions. We pose a few such problems, acknowledging that their difficulties could easily span a large gamut. 
\par 
To begin, we note that we have focused exclusively on rings $\mathcal O_{\mathbb{Q}(\sqrt{d})}$ with $d\!\in\! K$. One could generalize the definitions presented here to the other quadratic integer rings. While complications could surely arise in rings without unique factorization, generalizing the abundancy index to unique factorization domains $\mathcal O_{\mathbb{Q}(\sqrt{d})}$ with $d>0$ does not seem to be a highly formidable task. 
\par 
Even if we continue to restrict our attention to the rings $\mathcal O_{\mathbb{Q}(\sqrt{d})}$ with $d\!\in\! K$, we may ask some interesting questions. For example, for a given $n$, what are some examples of $n$-powerful friends in these rings? We might also ask which numbers (or which types of numbers), are $n$-powerfully solitary for a given $n$. For example, the number $21$ is solitary in $\mathbb{Z}$, so it is not difficult to show that $21$ is also solitary in $\mathcal O_{\mathbb{Q}(\sqrt{-1})}$. Furthermore, for a given element of some ring $\mathcal O_{\mathbb{Q}(\sqrt{d})}$, we might ask to find the values of $n$  for which this element is $n$-powerfully solitary. 
\begin{conjecture} \label{Conj4.1} 
Let $d\!\in\! K$. If $p$ is an integer prime and $k$ is a positive integer, then $p^k$ is $n$-powerfully solitary in $\mathcal O_{\mathbb{Q}(\sqrt{d})}$ for all positive integers $n$. As a stronger form of this conjecture, we wonder if $\pi^{\alpha_1}\overline{\pi}^{\alpha_2}$ is necessarily $n$-powerfully solitary in $\mathcal O_{\mathbb{Q}(\sqrt{d})}$ whenever $\pi$ is a prime in $\mathcal O_{\mathbb{Q}(\sqrt{d})}$ and $\alpha_1,\alpha_2,n\in\mathbb{N}$. Note that part $(e)$ of Theorem \ref{Thm2.2} and Theorem \ref{Thm3.2} resolve this issue for many cases.
\end{conjecture} 
\section{Acknowledgments} 
The author would like to thank Professor Pete Johnson for inviting him to the 2014 REU Program in Algebra and Discrete Mathematics at Auburn University and for making that program an extremely relaxed environment that proved exceptionally conducive to research. The author would also like to thank the unknown referee for his or her careful reading.

\end{document}